\documentclass{amsart}
\usepackage[utf8]{inputenc}

\usepackage{amsmath,amssymb,amsfonts}%
\usepackage{amsthm}%
\usepackage{latexsym,bm,mathtools}
\usepackage{url}
\usepackage{color}
\usepackage[foot]{amsaddr}
\usepackage{hyperref}

\newcommand{\R}{\mathbb R}
\newcommand{\N}{\mathbb N}
\newcommand{\T}{\mathbb T}
\newcommand{\id}{\mathrm{id}}
\newcommand{\im}{\mathrm{Im}}
\newcommand{\Diffmu}{\mathrm{Diff}_\mu}
\newcommand{\Diff}{\mathrm{Diff}}
\DeclareMathOperator{\grad}{\mathrm{grad}}
\DeclareMathOperator{\diver}{\mathrm{div}}
\DeclareMathOperator{\curl}{\mathrm{curl}}
\DeclareMathOperator{\sgrad}{\mathrm{sgrad}}

\newtheorem{theorem}{Theorem}

\newtheorem{proposition}[theorem]{Proposition}
\theoremstyle{definition}

\theoremstyle{remark}

\newtheorem{example}{Example}

\title{Conjugate points along Kolmogorov flows on the torus}

\author{Alice Le Brigant}
\address{SAMM, University Paris 1 Panthéon-Sorbonne, France}
\email{alice.le-brigant@univ-paris1.fr}

\author{Stephen C. Preston}
\address{Brooklyn College and CUNY Graduate Center, USA}
\email{stephen.preston@brooklyn.cuny.edu}

\begin{document}

\maketitle

\begin{abstract}
The geodesics in the group of volume-preserving diffeomorphisms (volumorphisms) of a manifold $M$, for a Riemannian metric defined by the kinetic energy, can be used to model the movement of ideal fluids in that manifold. The existence of conjugate points along such geodesics reveal that these cease to be infinitesimally length-minimizing between their endpoints. In this work, we focus on the case of the torus $M=\T^2$ and on geodesics corresponding to steady solutions of the Euler equation generated by stream functions $\psi=-\cos(mx)\cos(ny)$ for integers $m$ and $n$, called Kolmogorov flows. We show the existence of conjugate points along these geodesics for all 
pairs of strictly positive integers %
$(m,n)$, 
thereby completing the characterization of all pairs $(m,n)$ such that the associated Kolmogorov flow generates a geodesic with conjugate points.
\end{abstract}


\section{Introduction}

\subsection{Motivation}

Since Arnold in 1966~\cite{arnold2014differential}, the Euler equations for ideal fluids have had  a well-known geometric interpretation as geodesics on the group $\Diffmu(M)$  of volume-preserving diffeomorphisms, or volumorphisms, of a manifold $M$, under a right-invariant Riemannian metric defined by the kinetic energy. Since this approach can be made rigorous as in Ebin-Marsden~\cite{ebin1970groups} and establishes that the geodesic equation is actually a smooth ODE on the group of Sobolev $H^s$ volumorphisms on $M$ for $s>\tfrac{1}{2}\dim(M)+1$, it can be shown that the Riemannian exponential map is $C^{\infty}$, and is invertible near zero. This shows that volumorphisms sufficiently close in $H^s$ can be joined to the identity by a unique minimizing geodesic.
Since then similar geometric interpretations have been found for a variety of other PDEs of continuum mechanics; see Arnold-Khesin~\cite{MR4268535} for a survey and Misio{\l}ek-Preston~\cite{misioek2010fredholm} for an overview of the ODE approach.

On a sufficiently long time interval, a geodesic may cease to minimize the length between its endpoints (or equivalently, the energy for a constant-speed parameterization). When this happens, the endpoint is called a \emph{cut point}, and if in addition the geodesic ceases to be even infinitesimally minimizing between its endpoints, it is called a \emph{conjugate point}. The existence of conjugate points is intimately connected to the existence of positive-curvature sections along a geodesic. On the volumorphism group this existence was unclear since curvature computations have seemed to suggest $\Diffmu(M)$ has mostly negative curvature. The question of whether conjugate 
points exist on $\Diffmu(M)$ was posed already by Arnold~\cite{arnold2014differential} in 1966, but was not solved until Misio{\l}ek~\cite{misiolekstability} in 1993 found them along rotations on the space $\Diffmu(M)$ of volumorphisms of the 2- and 3-spheres $M=S^2$ and $M=S^3$. 

The first example on a flat $M$ was also found by Misio{\l}ek~\cite{misiolekconjugate}, who showed that along the Kolmogorov flow on the torus with stream function $\psi = -\cos{6x}\cos{2y}$, there is eventually a conjugate point. To do so he devised what is now known as the Misio{\l}ek criterion, giving a sufficient condition for existence of a conjugate point. This criterion has been successfully used to find conjugate points along other steady flows on other manifolds, such as the $2D$ ellipsoid~\cite{tauchi2022existence}, the $3D$ ellipsoid~\cite{lichtenfelz2022existence}, and the sphere~\cite{benn2021conjugate}. In three dimensions, conjugate points have a substantially different nature~\cite{ebinmisiolekpreston} and are much more common than in two dimensions, and can be found using a necessary and sufficient local criterion along any particle path~\cite{prestonfirst,preston2008wkb}, as shown by 
the second author. This technique was used in \cite{preston2017geometry} to find conjugate points along axisymmetric $3D$ flows.

It is known that the Misio{\l}ek criterion cannot capture \emph{all} conjugate points. For example Tauchi-Yoneda~\cite{tauchi2022existence} observed that the Misio{\l}ek index never detects conjugate points along the spherical rotation, in spite of the fact that infinitely many of them exist. More generally Tauchi-Yoneda~\cite{tauchi2022arnold} showed that the Misio{\l}ek criterion cannot detect conjugate points along an Arnold stable flow. A condition of the second author~\cite{preston2023conjugate} is more suitable for detecting conjugate points in such cases (and particularly along rotational flows).

Specifically on the torus, the problem of finding 
conjugate points along Kolmogorov flows $\psi = -\cos{mx} \cos{ny}$ for all positive integer pairs $(m,n)$ was posed by Drivas et al.~\cite{drivas2021conjugate}, generalizing from Misio{\l}ek's original example of $m=6$, $n=2$. (We will assume throughout that $m\ge n$ without loss of generality due to symmetry.) They found many additional examples using Misio{\l}ek's criterion, for pairs $(m,n)$ satisfying a condition equivalent to $n\ge 2$ and $m>\frac{3n^2+6}{\sqrt{3}n}$. Soon after the second author~\cite{preston2023conjugate} found via brute force search several more examples, including $(m,1)$ for $m\ge 2$, along with $(2,2)$, $(3,2)$, and $(3,3)$, and conjectured that conjugate points existed for all $(m,n)$ except possibly for $(1,1)$.

\subsection{Contributions}

In this paper we prove this conjecture, that conjugate points exist for \emph{all strictly positive} integer pairs $(m,n)$, 
giving an explicit form of a test function for 
$m>n$, and a slightly different test function for $m=n\ge 2$, 
along with an example due to Drivas in the case $m=n=1$. In the case $n=0$ with $m>0$, it is known~\cite{misiolekstability} that there are no conjugate points. Hence we have obtained a 
complete characterization of all pairs $(m,n)\subset \mathbb{Z}_{\ge 0}^2$ such that the Kolmogorov flow $\psi = -\cos{mx}\cos{ny}$ generates a geodesic in the volume-preserving diffeomorphism group that infinitesimally minimizes length between its endpoints for all time: if $m$ and $n$ are both nonzero, 
the geodesic cannot minimize; if $m$ or $n$ is zero then the geodesic infinitesimally minimizes. 

\subsection{Outline}
In Section \ref{background} we describe the geometric approach to the Euler equation of ideal fluids as originally pioneered by Arnold. We also review Kolmogorov flows, which may be generally defined as those steady solutions of the Euler equation with stream function $\psi$ satisfying $\Delta \psi = -\lambda^2 \psi$, and specifically on the torus of the form $\psi =-\cos{mx}\cos{ny}$. Finally we recall the index form for detecting conjugate points and the Misio{\l}ek criterion which greatly simplifies this computation. In Section \ref{mainresults} we prove Theorem \ref{offdiagonal} (the $m>n$ case), Theorem \ref{diagonal} (the $m=n\geq 2$ case)  
and Theorem \ref{diagonal11thm} (the $m=n=1$ case),
showing the existence of conjugate points.
In Section \ref{minimization} we discuss the relationship between the Misio{\l}ek index and the Rayleigh quotient restricted to a closed subspace. Section \ref{numerics} describes the algorithm we devised to exploit this relationship in order to construct an optimal variation field numerically; those given in Section \ref{mainresults} are simply truncated versions of these that still work. 
In this section we explain why all of the variation test functions we find have the same basic shape: a perturbation of the simplest Laplacian eigenfunction $\cos{x}$ or $\sin{x}$. 
Finally in Section \ref{11sec} we 
suggest some other problems about conjugate points that one can tackle using the same methods as those presented here. In an Appendix we present a simple formula that is useful numerically when writing the Poisson bracket in a Fourier basis, necessary for using standard matrix algorithms to optimize the Misio{\l}ek index.

\subsection{Computations}

Symbolic computations were performed in Maple 2021, while numerical computations were performed using Python. Maple code for computing the index form as in Theorems \ref{offdiagonal}--\ref{diagonal11thm}, and Python code for numerically minimizing the index form and obtaining the form of the candidate minimizers, are both available on github: \url{https://github.com/alebrigant/conjugate-points}.

\subsection{Acknowledgements}

The authors acknowledge support of the Institut Henri Poincar{\'e} (IHP, UAR 839 CNRS-Sorbonne Université), and LabEx CARMIN (ANR-10-LABX-59-01). This work was done while the second author visited the first author at IHP for the Geometry and Statistics in Data Sciences (GESDA) thematic quarter, funded by the French National Center for Scientific Research (CNRS). Both authors thank the IHP for their hospitality. 
We also thank Theo Drivas and Alexander Shnirelman for very helpful discussions and suggestions.

\section{Background}\label{background}

\subsection{Volumorphisms and the Euler-Arnold equation}

Suppose $M$ is a $2$-dimensional manifold equipped with a Riemannian metric $\langle\cdot,\cdot\rangle$ inducing an area form $\mu$. We are interested in 
diffeomorphisms $\varphi:M\rightarrow M$ that preserve the area form: $\varphi^*\mu=\mu$, called volume-preserving diffeomorphisms of $M$ or volumorphisms for short. The space $\Diffmu(M)$ of volumorphisms of $M$ is (formally) a submanifold of the space of diffeomorphisms $\Diff(M)$. The tangent vectors at $\varphi\in\Diffmu(M)$ are right translations $X\circ\varphi$ of divergence-free vector fields $X$ of $M$ tangent to the boundary. We refer the interested reader to \cite{ebin1970groups} for more details on these manifold structures in the context of Sobolev spaces.

A volumorphism $\varphi$ can be seen as describing the positions, at a given time, of the particles of an ideal fluid (incompressible and inviscid) moving inside of $M$: the value $\varphi(x)$ gives the position of the particle that was at position $x\in M$ at $t=0$. The volume-preserving property of $\varphi$ is a consequence of the incompressibility of the fluid. With this interpretation, a tangent vector $X\circ \varphi$ is the velocity field of the fluid at that time, and the kinetic energy of the fluid defines a Riemannian metric on $\Diffmu(M)$
\begin{equation}\label{kinetic_metric}
g_\varphi(X\circ\varphi, Y\circ\varphi)=\int_M\langle X\circ\varphi,Y\circ\varphi\rangle\mu.
\end{equation}
Since for any volume-preserving $\varphi\in\Diffmu(M)$, the pullback measure is $\varphi^*\mu=\mu$, we see that
$$g_\varphi(X\circ\varphi, Y\circ\varphi)=\int_M \langle X,Y\rangle\mu=g_{\id}(X, Y),$$
and so the kinetic metric is right-invariant on $\Diffmu(M)$. 

The result shown by Arnold \cite{arnold2014differential} is the following: the geodesics in $\Diffmu(M)$ for the kinetic metric \eqref{kinetic_metric} describe the motion of an ideal fluid in $M$. To see this, we can consider the geodesic equation on the larger space $\Diff(M)$, which is simply $\partial^2\gamma/\partial t^2=0$, and then orthogonally project it on the submanifold $\Diffmu(M)$. In terms of the velocity field $X(t,\cdot)$ associated with the geodesic $\gamma(t,\cdot)$ 
$$\frac{\partial \gamma}{\partial t}(t,x)=X(t,\gamma(t,x)),$$ 
this gives
$$P\left(\frac{\partial^2\gamma}{\partial t^2}\right)=P\left(\frac{\partial X}{\partial t}+\nabla_XX\right)\circ\gamma=0,$$
where $\nabla$ is the Levi-Civita connection associated to the Riemannian metric on $M$. The orthogonal projection $P$ on the space of divergence-free vector fields tangent to the boundary is obtained by the Hodge decomposition of vector fields
\begin{equation}\label{proj}
P(X)=X-\grad f,
\end{equation}
where $f$ is a function verifying $\Delta f=\diver X$ and whose normal component along the boundary equals that of the vector field $X$. The geodesic equation for the kinetic metric \eqref{kinetic_metric} on $\Diffmu(M)$ is therefore given by
\begin{equation}\label{euler-arnold}
    \frac{\partial X}{\partial t}+\nabla_XX=-\grad p,
\end{equation}
where $p$ is called the pressure function, defined up to a constant by $\Delta p=-\diver(\nabla_XX)$ and its component along the normal $\nu$: $\langle \grad p,\nu\rangle=-\langle X, \nu\rangle$. Equation \eqref{euler-arnold} is the Euler equation for incompressible fluids, also called the Euler-Arnold equation in virtue of its interpretation by Arnold as a geodesic equation on the space of volumorphisms.

\subsection{Kolmogorov flows on the torus}

In the two-dimensional case, it can be shown that the Euler-Arnold equation \eqref{euler-arnold} can be rewritten in terms of the curl as
$$\frac{\partial}{\partial t}(\curl X)+X(\curl X)=0.$$
Now divergence-free vector fields on surfaces can be written as $X=\sgrad\psi + W$ where $\psi:M\rightarrow\R$ is defined by $\Delta\psi = \curl X$ and $\left.\psi\right|_{\partial M}=0$, while the vector field $W$ satisfies $\curl W=0$ and $\diver W = 0$. For such a vector field $X$, the Euler-Arnold equation becomes
$$\frac{\partial}{\partial t}\Delta \psi+\{\psi,\Delta\psi\}+W(\Delta\psi)=0.$$
Here $\{\cdot,\cdot\}$ is the Poisson bracket, which can be defined by the following property of Hamiltonian vector fields: if $\sgrad f$ and $\sgrad g$ are Hamiltonian vector fields generated by mean-zero functions $f,g:M\rightarrow\R$, then $\{f,g\}$ is the unique mean-zero function generating their Lie bracket
$$[\sgrad f,\sgrad g]=\sgrad\{f,g\}.$$
In dimension 2, the Poisson bracket is simply given by 
$$\{f,g\}\mu= df\wedge dg,$$
which reduces to $\{f,g\}=\partial_xf\partial_yg-\partial_yf\partial_xg$ on the torus. We then see that any 
function $\psi$ for which 
\begin{equation}\label{steadystream}
\{ \psi, \Delta \psi\} = 0,
\end{equation}
generates a steady solution $X = \sgrad \psi$ of the Euler-Arnold equation, i.e., a geodesic $\gamma$ with associated velocity field $X$ satisfying $\frac{\partial X}{\partial t}=0$. The function $\psi$ is called the \emph{stream function}. Here we will focus on geodesics on $\Diffmu(\T^2)$ whose velocity field is generated by the stream functions
\begin{equation}\label{stream-fctn}
    \psi(x,y)=-\cos(mx)\cos(ny), \quad (m,n)\in\N^2.
\end{equation}
These are eigenfunctions of the Laplacian for the eigenvalues $-\lambda^2=-(m^2+n^2)$, and they satisfy \eqref{steadystream} to give steady solutions of the Euler equations, the so-called Kolmogorov flows on the torus $M=\T^2$. For these flows, the fluid has a circular movement inside rectangular grid cells, as shown in Figure~\ref{fig:streamlines}.
The question that we ask is: can we find conjugate points along any Kolmogorov flow on the torus? The existence of conjugate points is related to the question of the uniqueness of the geodesic between its endpoints, and whether a perturbation of the initial condition of the corresponding flow can lead to the same result as no perturbation at all. 

\begin{figure}
    \centering
    \includegraphics[width=0.95\linewidth]{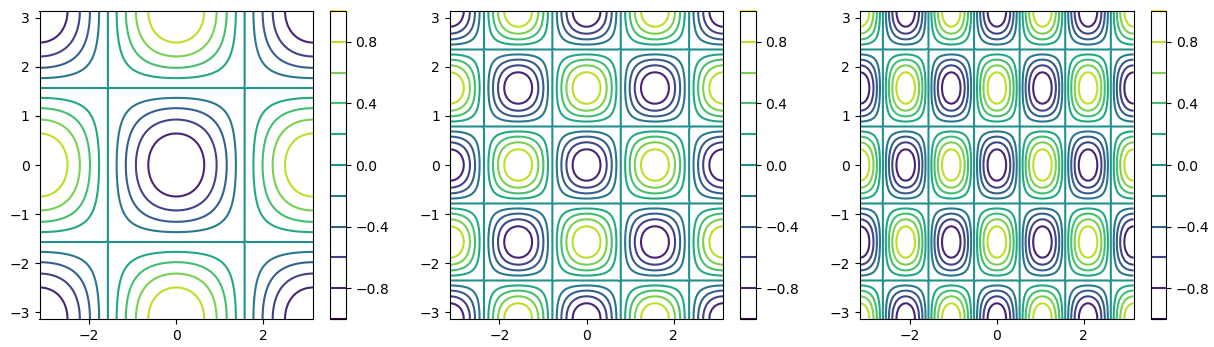}
    \caption{Kolmogorov flows on the torus: level sets of the stream function $\psi(x,y)=-\cos(mx)\cos(ny)$ for $m=n=1$ (left), $m=n=2$ (middle), $m=3$, $n=2$ (right). The velocity of the fluid is tangent to these level sets.}
    \label{fig:streamlines}
\end{figure}

\subsection{Conjugate points and the Misio{\l}ek criterion}

Two volumorphisms $\varphi_1$ and $\varphi_2$ are conjugate if there exists a family of geodesics that start at $\varphi_1$ and end at $\varphi_2$ up to first order; 
more precisely, if there exists $\gamma(s, t)$ with $(s,t)\in(-\epsilon,\epsilon)\times[0,1]$ such that for all $s$, $\gamma(s,\cdot)$ is a geodesic in $\Diffmu(M)$, $\gamma(s,0)=\varphi_1$, $\gamma(s,1)=\varphi_2$, 
and the corresponding Jacobi field $J(t)=\frac{\partial\gamma}{\partial s}(0,t)$ satisfies $J(0)=J(1)=0$.
We say that there is a conjugate point along a geodesic $\gamma$ if there exists a time $T>0$ such that $\gamma(0)$ and $\gamma(T)$ are conjugate points. Conjugate points are particular cases of cut points. However, they can only happen in manifolds with some positive curvature, while cut points can occur even in flat manifolds due to non trivial topology. 

Conjugate points can be detected using the so-called index form, defined for any vector field $Y(t)$ along the geodesic $\gamma(t)$ by
$$I(Y,Y)=\int_0^{T}\left\langle\frac{DY}{dt}, \frac{DY}{dt}\right\rangle - \left\langle R(Y,\dot\gamma)\dot\gamma,Y\right\rangle \,dt.$$
It can be shown~\cite{misiolekstability} that $I(Y,Y) = 0$ for some $Y$ satisfying $Y(0) = Y(T) = 0$ if and only if $Y$ is a Jacobi field. Furthermore if $I(Y,Y) < 0$ for some such field $Y$, then there is a Jacobi field $J$ along $\gamma$ vanishing at $t = 0$ and for $t = \tau$ for some $\tau <T$. Establishing negativity of the index form for some vector field $Y$ is thus an effective way to show there must be a conjugate point without actually having to find it. In our setting of volume-preserving diffeomorphisms on a surface, the index form can be written in the following way.
\begin{proposition}[Corollary 9 in \cite{preston2023conjugate}]
Suppose $M$ is a surface
and that $X = \sgrad \psi$ solves the Euler equation \eqref{euler-arnold} on $[0,T]$. For a time-dependent family of functions $g(t)$ on $M$, vanishing on the boundary of $M$ and at $t=0$ and $t=T$, the index form for $Y =\sgrad g$ along the geodesic with velocity field $X$ becomes
\begin{equation}\label{index-form}
I(Y,Y)=\int_0^{T}\int_M |\grad h|^2 + \Delta \psi \{g,h\} \,d\mu dt, \quad
h:=\frac{\partial g}{\partial t} + \{\psi,g\}.
\end{equation}
\end{proposition}
 The Misio{\l}ek criterion consists of computing the index form for a particular family of deformations, namely $g(t)=\sin(\frac{\pi t}{T})f$ for some function $f:M\rightarrow\R$, and then letting time $T$ go to infinity. If the result is negative for some function $f$, then for sufficiently large $T$ the index form will be negative and hence there will be a conjugate point occurring at some time $\tau <T$. 


Consider a steady geodesic in $\Diffmu(M)$ generated by a stream function $\psi:M\rightarrow\R$ satisfying 
$\Delta \psi =\Upsilon\circ\psi$ for some function $\Upsilon\colon \mathbb{R}\to\mathbb{R}$; every steady solution of \eqref{steadystream} can at least locally be expressed in this form.
The index form \eqref{index-form} at a vector field $Y=\sgrad g$ generated by a family of deformations $g(t)=\sin(\frac{\pi t}{T})f$ for a given function $f$ can be written in terms of the Poisson bracket $\phi = \{\psi, f\}$ as
$$I(g)=\frac{\pi^2}{2T}\int_M |\grad f|^2 d\mu + \frac{T}{2}\int_M \Big(|\grad\phi|^2  
+ \Upsilon'(\psi) \,\phi^2\Big) d\mu.$$
The Misio{\l}ek criterion states that there is a conjugate point eventually occurring along the geodesic if 
$$\lim_{T\to\infty} \frac{2}{T} I(g)<0$$
for some function $f$, i.e.,
$$
MI(\phi) := \int_M \Big(\lvert \grad \phi\rvert^2 
+\Upsilon'(\psi)\,
\phi^2 \Big) \, d\mu < 0, \qquad
\phi=\{\psi,f\}.
$$
In the present case, for $M=\mathbb{T}^2$, we will take 
\begin{equation}\label{psilambda}
\psi=-\cos{mx}\cos{ny}, \qquad \text{with}\quad 
\Upsilon(\psi) = -\lambda^2 \psi, \quad\text{where } 
\lambda^2 = m^2+n^2.
\end{equation}
This leads to the explicit version of the Misio{\l}ek index in the form 
\begin{equation}\label{misiolek_index}
MI(\phi) := \int_M \Big(\lvert \grad \phi\rvert^2 
- (m^2+n^2) \phi^2 \Big) \, d\mu < 0, \qquad
\phi=\{\psi,f\},
\end{equation}
which is the version we will use in the remainder of the paper. We will show that when $\phi=-\cos{mx}\cos{ny}$ for positive integers $m,n$, there is always a function $f$ to make $MI(\phi)<0$, leading to both existence of conjugate points and the indefiniteness of Arnold's quadratic form.

\subsection{Remark on the link with stabillity}
Note that the Misio{\l}ek index is precisely the second variation of the Dirichlet energy $\int_M \lvert \grad \psi\rvert^2 \, d\mu$, which is the $L^2$ energy of a vector field, when considered under the action by volume-preserving diffeomorphisms of a $2D$ manifold. Here variations are generated by test functions $f$, while the actually varied stream function under the action is given by $\phi=\{\psi,f\}$. The formula for the second variation of the energy is given in Chapter II, Remark 2.4 of Arnold-Khesin~\cite{MR4268535}, and it agrees with the Misio{\l}ek index. Hence the second variation of energy of this minimization problem is indefinite if and only if there is a conjugate point that is detectable by the Misio{\l}ek criterion. As pointed out by Tauchi-Yoneda~\cite{tauchi2022arnold}, this implies that conjugate points along an Arnold-stable flow cannot be detected by the Misio{\l}ek criterion, although they may exist (as happens for many rotational flows, such as on the sphere; see \cite{preston2023conjugate} for examples).

Indefiniteness of the second variation does not necessarily imply instability of the steady flow, however, and thus existence of conjugate points detectable by the Misio{\l}ek criterion does not necessarily imply that the linearized Euler equation has unstable growing solutions. See Arnold-Khesin, Chapter II, Remark 3.7~\cite{MR4268535}. It would be very interesting to further clarify the apparent connection between \emph{instability} of the linearized Euler equation and the existence of conjugate points along the corresponding geodesic, particularly since the Arnold criterion for stability and the Misio{\l}ek criterion for conjugate points are both only sufficient but not necessary, and since the more intuitive connection would be between conjugate points and \emph{stability}.

Note that the second variation formula (and thus the Misio{\l}ek criterion) apply to finding infinitesimal energy-decreasing variations of the energy by diffeomorphisms, in terms of the local analysis of vector fields at the identity. The problem of finding the global minimizer was studied by Shnirelman and Lan~\cite{Lan2012TheMO}, who numerically found a minimizer that is not smooth in the case $(m,n)=(1,1)$. It is possible that the roughness of this optimizer is related to some difficulty in the second variation we are studying here, but the connection is not completely clear.

\section{Main results}\label{mainresults}


We will show that for every pair $(m,n)$ of positive integers with $m\ge n$, the Misio{\l}ek criterion \eqref{misiolek_index} detects a conjugate point along the geodesic in $\Diffmu(\mathbb{T}^2)$.
It is easier to handle the three cases $m>n$,  $m=n\ge 2$, and $m=n=1$ separately.

The formulas presented in the next three theorems were discovered by numerics we will elaborate on in Section \ref{numerics}. They are finite truncations of infinite Fourier series, dominated by the lowest-order term (either $\cos{x}$ or $\sin{x}$), which show up because they are the eigenfunctions of the Laplacian with smallest eigenvalue $-1$.

\begin{theorem}\label{offdiagonal}
If $m$ and $n$ are positive integers with $m>n$, and $\psi(x,y)=-\cos{mx}\cos{ny}$, then there is a function $f\colon \mathbb{T}^2\to\mathbb{R}$ such that $\phi=\{\psi,f\}$ satisfies the Misio{\l}ek criterion \eqref{misiolek_index}. Hence if the geodesic $\gamma$ in $\Diffmu(\mathbb{T}^2)$ has initial conditions $\gamma(0)=\id$ and $\gamma'(0)=\sgrad \psi$, then $\gamma(T)$ is conjugate to $\gamma(0)$ for some $T>0$, and $\gamma$ is not minimizing beyond $T$.
\end{theorem}

\begin{proof}
See Maple file in \url{https://github.com/alebrigant/conjugate-points} for details of the computations here.

For numbers $a,b\in\mathbb{R}$, define 
\begin{equation}\label{foffdiag}
f(x,y) = \cos{x}\big(1+a\cos{(2mx)}+ b\cos{(2ny)}\big).
\end{equation}
We compute the Poisson bracket
\begin{align*}
\phi(x,y) &= \psi_x(x,y)f_y(x,y)-\psi_y(x,y)f_x(x,y) \\
&= 2mn\cos x\big( a\cos(mx)\sin(ny)\sin(2mx) - b \sin(mx)\cos(ny)\sin(2ny) \big) \\
&+ n\cos(mx)\sin(ny) \sin{x}\big(1+a\cos(2mx)+ b\cos(2ny)\big). 
\end{align*}

The quantity $MI$ from \eqref{misiolek_index} is then given by 
$$MI = \frac{\pi^2n^2}4 H(a,b,m,n),$$
where
\begin{multline}\label{Hdef}
H(a,b,m,n)=16 a^2 m^4+16 b^2 m^2 n^2+24 a^2 m^2-12 a b m^2+4 b^2 m^2+4 b^2 n^2 \\
+8 a m^2-8 b m^2+a^2-a b+b^2+2 a-2 b+2. 
\end{multline}
This is quadratic in both $a$ and $b$, with positive leading-order coefficients, and thus the critical point must be a global minimizer. However the formula is slightly complicated, and it is easier to use the points 
\begin{equation}\label{aboffdiag}
a_0=-\frac{4m^2+1}{16m^4+24m^2+1} \qquad \text{and}\qquad b_0=\frac{1}{4n^2+1};
\end{equation}
here $a_0$ is the minimizer of $H(a,0,m,n)$ and $b_0$ is the minimizer of $H(0,b,m,n)$. With these choices we obtain
\begin{equation}\label{HJdefoffdiag}
\begin{split}
H(a_0,b_0,m,n) &= \frac{J(m,n)}{(16m^4+24m^2+1)(4n^2+1)}, \text{  where }\\
J(m,n) &= 4 n^2 (16 m^4 + 40 m^2 + 1)  - 64m^6 - 48m^4 + 28m^2 + 1.
\end{split}
\end{equation}
We want to show that $J(m,n)<0$ for every choice of naturals $m$ and $n$ with $m>n$. Clearly 
$J$ is increasing as a function of $n$, so we have 
$$ J(m,n) \le J(m,m-1) = -128 m^5 + 176 m^4 - 320 m^3 + 192 m^2 - 8 m + 5 \quad \text{for $m\ge 2$}.$$
Writing $m=k+1$ for $k\ge 1$ this becomes 
$$ J(m,n) \le J(k+1,k) = -128 k^5 - 464 k^4 - 896 k^3 - 992 k^2 - 520 k - 83,$$
which is obviously negative. 


We conclude that $J(m,n)<0$ for all positive integers $m,n$ with $m>n$. This ensures that $H(a_0,b_0,m,n)<0$, and thus that $MI<0$ for these choices of $a_0$ and $b_0$. By the Misio{\l}ek criterion, there is eventually a conjugate point along the corresponding geodesic.
\end{proof}

Observe in the proof above that the worst-case scenario is when $n=m-1$: when $n$ is smaller than $m-1$ the index form is even more negative. This corresponds to conjugate points being easier to find when $m$ and $n$ are farther apart, as was found by Drivas et al.~\cite{drivas2021conjugate}. It thus stands to reason that the hardest case is when $n=m$. Indeed we need two extra terms in our formula for the variation to make it work in this case. Note also that we require that $n\ge 2$; this variation field does not work if $m=n=1$. We will handle that case separately.

\begin{theorem}\label{diagonal}
Suppose $n\ge 2$ and $\psi(x,y) = -\cos{(nx)}\cos{(ny)}$. Then there is a function $f\colon \mathbb{T}^2\to\mathbb{R}$ such that $\phi=\{\psi,f\}$ satisfies \eqref{misiolek_index}. Thus there is a time $T>0$ such that $\gamma(T)$ is conjugate to $\gamma(0)$ along the corresponding geodesic $\gamma$, and $\gamma$ is not minimizing past $T$. 
\end{theorem}

\begin{proof}
See Maple file in \url{https://github.com/alebrigant/conjugate-points} for details of the computations here.

For coefficients $a,b,c,d\in\mathbb{R}$, define 
\begin{equation}\label{fdiag}
f(x,y) = \cos{x} \big(1+a \cos{(2ny)}+b\cos{(4ny)}+c\cos{(2nx)}\big)+d\sin{x} \sin{(2nx)}.
\end{equation}
We again compute 
\begin{multline}
\phi(x,y) 
=n \sin{(n x)} \cos{(n y)} \cos{x} \Big(-2 a n \sin{(2 n y)}-4 b n \sin{(4 n y)}\Big) \\
- n \cos{(n x)} \sin{(n y)} \Big(-\sin{x} \big[(1+a \cos{(2 n y)}+b \cos{(4 n y)}+c \cos{(2 n x)} \big]\\
-2cn \cos{x} \sin{(2 n x)}+d \cos{x} \sin{(2 n x)}+2 dn \sin{x} \cos{(2 n x)}\Big)
\end{multline}
Computing the Misio{\l}ek index \eqref{misiolek_index} now gives 
$$ MI = \frac{\pi^2 n^2}4 H(a,b,c,d,n),$$
where 
\begin{multline}\label{Hdiag}
H(a,b,c,d,n) = 16a^2n^4 + 64abn^4 + 256b^2n^4 + 16c^2n^4 + 16d^2n^4 + 8adn^3 \\ - 64cdn^3
+ 8a^2n^2 - 4abn^2 - 12acn^2 + 32b^2n^2 + 24c^2n^2 + 24d^2n^2 + 6adn - 8an^2\\ - 16cdn + 8cn^2
+ a^2 - ab - ac + b^2 + c^2 + d^2 - 8dn - 2a + 2c + 2.
\end{multline}
Again, this is quadratic in the four unknown coefficients $a,b,c,d$ with positive leading-order coefficients, so the unique critical point is a global minimum. 

In this case there is no simpler formula which works, so we find the critical point by solving the linear system 
$$ \frac{\partial H}{\partial a} = \frac{\partial H}{\partial b} = \frac{\partial H}{\partial c} =\frac{\partial H}{\partial d} = 0.$$
Explicitly this system looks like 
\begin{align*}
(32a + 64b)n^4 + 8dn^3 + (16a - 4b - 12c - 8)n^2 + 6dn + 2a - b - c - 2 &= 0, \\
(64a + 512b)n^4 + (-4a + 64b)n^2 - a + 2b &= 0, \\
32cn^4 - 64dn^3 + (-12a + 48c + 8)n^2 - 16dn - a + 2c + 2 &= 0, \\
32dn^4 + (8a - 64c)n^3 + 48dn^2 + (6a - 16c - 8)n + 2d &= 0.
\end{align*}

There is a unique solution of the system above, given by 
\begin{align*}
a_0 &= \frac{(8 n^2+1) (256 n^4+32 n^2+1)}{6144 n^8+4864 n^6+920 n^4+58 n^2+1}, \\
b_0 &= -\frac{512 n^6+32 n^4-12 n^2-1}{2 (6144 n^8+4864 n^6+920 n^4+58 n^2+1)}, \\
c_0 &= -\frac{49152 n^{10}+59392 n^{8}+17088 n^6+1952 n^4+88 n^2+1}{2 (98304 n^{12}+28672 n^{10}-18048 n^8-1568 n^6+472 n^4+50 n^2+1)}, \\
d_0 &= -\frac{n (32768 n^{8}+19456 n^6+3328 n^4+196 n^2+3)}{(16 n^4-8 n^2+1) (6144 n^8+4864 n^6+920 n^4+58 n^2+1)}.
\end{align*}
In spite of how complicated these formulas appear, the value of $H$ at the critical point is relatively simple in $n$: we get 
\begin{equation}\label{Hmindiag}
H(a_0,b_0,c_0,d_0,n) = \frac{-4096 n^{8}+3584 n^{6}+1008 n^{4}+68 n^{2}+1}{12288 n^{8}+9728 n^{6}+1840 n^{4}+116 n^{2}+2}.
\end{equation}
Replacing $n$ with $\sqrt{4+k}$ for some $k\ge 0$ (since $n\ge 2$ by assumption), this becomes 
$$ H(a_0,b_0,c_0,d_0,\sqrt{4+k}) 
= \frac{-4096 k^{4}-61952 k^{3}-349200 k^{2}-868412 k-802799}{12288 k^{4}+206336 k^{3}+1298224 k^{2}+3627508 k+3798226},
$$
and this is obviously negative for all real $k\ge 0$. 

We conclude that with $n\ge 2$, the choice $(a,b,c,d)=(a_0, b_0,c_0,d_0)$ leads to a variation $f$ given by \eqref{fdiag} such that $\phi=\{\psi, f\}$ satisfies the Misio{\l}ek criterion \eqref{misiolek_index}. Hence there is eventually a conjugate point along the geodesic.
\end{proof}

\begin{example}
In the case $m=3$ and $n=2$, the formula \eqref{foffdiag} with $a$ and $b$ given by \eqref{aboffdiag} becomes 
\begin{equation}\label{zeta32}
f(x,y) = \cos{x} \left(1-\frac{37}{1513} \, \cos{6x} +\frac{1}{17} \, \cos{4y}\right).
\end{equation}
This is plotted on the left side of Figure \ref{diagperturbfig} below.

On the other hand if $m=n=2$, the variation $f$ from \eqref{fdiag} becomes 
\begin{multline}\label{fdiag22}
f(x,y) = \cos{x} \left(1+\frac{139425}{1899113}\,\cos{4 y}-\frac{33231}{3798226} \,\cos{8 y}-\frac{66661217}{854600850} \, \cos{4 x} \right) \\
-\frac{19375654}{427300425}\,\sin{x} \sin{4 x}.
\end{multline}
The plot is shown on the right side of Figure \ref{diagperturbfig}. Note that especially in the off-diagonal case, both the formula and graph are substantially simpler than the one given in the second author's paper~\cite{preston2023conjugate}. Also note that the graphs are basically indistinguishable to the naked eye, although none of the terms in \eqref{fdiag22} can be omitted without changing the sign of the index. 
\end{example}
 
\begin{figure}[!ht]
\begin{center}
\includegraphics[width=0.4\linewidth]{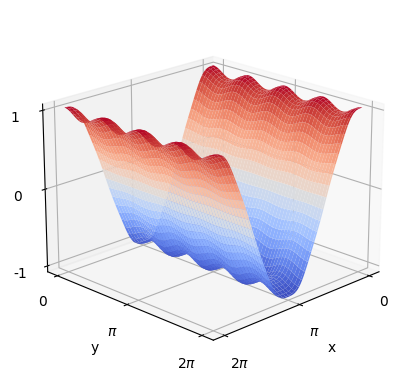}
\includegraphics[width=0.4\linewidth]{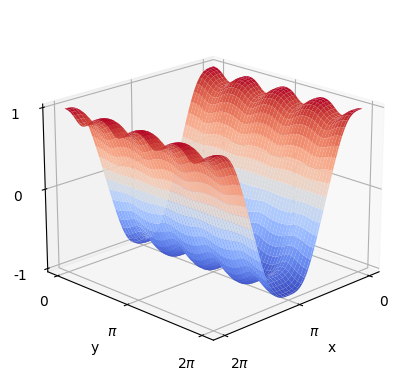}
\caption{On the left, the optimal perturbation $f$ of the form \eqref{foffdiag} in the case $m=3$ and $n=2$, given explicitly by \eqref{zeta32}. On the right, the optimal perturbation $f$ of the form \eqref{fdiag}, given explicitly by \eqref{fdiag22}, in the case $m=n=2$. These perturbations generate a variation $\phi$ satisfying the Misio{\l}ek criterion \eqref{misiolek_index} for conjugate points. Note that they appear quite similar at this level.}\label{diagperturbfig}
\end{center}
\end{figure}

Variation fields of the form considered in Theorems \ref{offdiagonal} and \ref{diagonal}, dominated as they are by the $\cos{x}$ term, do not seem suitable for making the Misio{\l}ek index negative in the case $m=n=1$. Instead the following field dominated by $\sin{x}$, found by Theodore Drivas (personal communication) and included here with his permission, gives an energy-reducing variation in that case.

\begin{theorem}\label{diagonal11thm}
Suppose $\psi(x,y) = -\cos{x}\cos{y}$. Then there is a function $f\colon \mathbb{T}^2\to\mathbb{R}$ such that $\phi=\{\psi,f\}$ satisfies \eqref{misiolek_index}. Thus there is a time $T>0$ such that $\gamma(T)$ is conjugate to $\gamma(0)$ along the geodesic $\gamma$, and $\gamma$ is not minimizing past $T$.
\end{theorem}

\begin{proof}
Define 
$$ f(x,y) = \sin{x} + \tfrac{1}{10} \sin{(x+2y)} - \tfrac{1}{20} \sin{(3x)} + \tfrac{1}{100} \sin{(5x)}.$$
We easily compute that the Poisson bracket $\phi=\{\psi,f\}$ is
\begin{multline*} \phi(x,y) = 
\tfrac{1}{5} \sin{x} \cos{y} \cos{(x+2y)}  \\ -\cos{x} \sin{y} \big( \cos{x} + \tfrac{1}{10} \cos{(x+2y)} - \tfrac{3}{20} \cos{(3x)} + \tfrac{1}{20} \cos{(5x)}\big). 
\end{multline*}

The Misio{\l}ek index \eqref{misiolek_index} is given by 
$$ MI = \iint_{\mathbb{T}^2} \lvert \nabla \phi\rvert^2 - 2\phi^2 \, dA = -\frac{3\pi^2}{200},$$
and we are done.
\end{proof}

\begin{figure}
    \centering
    \includegraphics[width=0.95\linewidth]{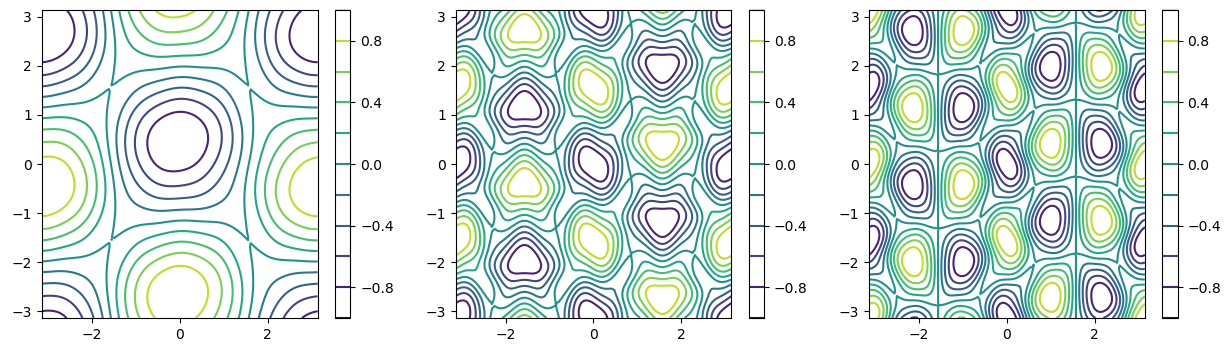}
    \caption{Local deformation of the stream function $\psi$ induced by the area-preserving diffeomorphism $\mathrm{sgrad}f$, i.e., the function $\psi(x-\epsilon \partial_yf, y+\epsilon\partial_xf)$, for $m=n=1$ (left), $m=n=2$ (middle) and $m=3$, $n=2$ (right).}
    \label{fig:stream-deformed}
\end{figure}

In Figure~\ref{fig:stream-deformed} we show the local deformation of the stream function $\psi$ induced by the area-preserving diffeomorphism generated by the test function $f$, in an example of each of the three cases discussed above. This can be compared to Figure~\ref{fig:streamlines} which shows $\psi$ without deformation. The second author~\cite{preston2023conjugate} showed that any test function solely supported in a single cell of a Kolmogorov flow on $\mathbb{T}^2$ cannot generate an energy-reducing variation, even if time variation is allowed. Hence one intuitively expects that the rectangular gridlines must be substantially deformed by any energy-reducing variation, which is indeed what we see. Unsurprisingly the most deformation happens at the vertices where gridlines intersect, since there the fluid has hyperbolic fixed points and the fluid is highly susceptible to small perturbations.


\section{The minimization principle}\label{minimization}

In this section we relate the Misiolek criterion to an eigenfunction problem on the Poisson bracket $\phi=\{\psi,f\}$. Recall that the Misio{\l}ek criterion  \eqref{misiolek_index} is given by 
\begin{equation}\label{misiolekindexagain}
MI(\phi) := \int_M \lvert \grad \phi\rvert^2 - \lambda^2 \phi^2 \, d\mu 
< 0, 
\end{equation}
where $\phi = \{\psi,f\}$ for some function $f$.

\begin{theorem}\label{eigenvaluetheorem}
If $\psi\colon M\to \mathbb{R}$ is a fixed smooth function satisfying $\Delta \psi = -\lambda^2 \psi$ and $L$ denotes the operator $L(f) = \{\psi,f\}$, then 
$\Delta^{-1}L$ is continuous from $\dot{H}^1(M)$ to itself. The image $\im[L]$ is a closed subspace of $\dot{H}^1(M)$, and $\Delta^{-1}$
restricts to a continuous operator on it. 
The Misio{\l}ek criterion is satisfied if and only if the operator norm of $\Delta^{-1}$ on $\im[L]$ satisfies
\begin{equation}
    \lVert \Delta^{-1}\big|_{\im[L]}\rVert_{\text{op}} \ge \frac{1}{\lambda^2}.
\end{equation}
This happens if and only if there is a number $c\ge 1/\lambda^2$ and a function $f\colon M\to\mathbb{R}$ satisfying 
$$ -\Delta^{-1} L(f) = cL(f). $$
\end{theorem}

\begin{proof}
Recall that the first-order homogeneous Sobolev metric is given by
$$\langle f,g\rangle_{\dot H^1}=-\int_M f\Delta g \,d\mu = \int_M \langle \nabla f,\nabla g\rangle \, d\mu.$$
To establish continuity it is sufficient to establish boundedness, which follows from 
$$ \lVert L(f)\rVert_{L^2} = \lVert \grad \psi\times \grad f\rVert_{L^2} \le \lVert \psi\rVert_{C^1} \lVert \grad f\rVert_{L^2} = 
\lVert \psi\rVert_{C^1} \lVert f\rVert_{\dot{H}^1}.$$
As a result the image 
of $L$ is a closed subspace in $\dot{H}^1(M)$. Since $L$ is antisymmetric in the $L^2$ inner product 
we see that the operator $\Delta^{-1}L$ is bounded and antisymmetric in $\dot{H}^1$, via 
\begin{multline*} 
\langle \Delta^{-1}L(f), g\rangle_{\dot{H}^1(M)} = -\int_M L(f)g\,d\mu = 
-\int_M \{\psi, f\} g\, d\mu \\
= \int_M \{\psi, g\} f \, d\mu = -\langle \Delta^{-1}L(g), f\rangle_{\dot{H}^1(M)}.
\end{multline*}


Now the Misio{\l}ek index from \eqref{misiolekindexagain} rescaled by the $\dot H_1$ norm is given by 
$$ \frac{MI(\phi)}{\lVert \phi\rVert^2_{\dot{H}^1(M)}} = 
1 -  \lambda^2  \frac{\int_M \phi^2 \, d\mu}{\int_M \lvert \grad \phi\rvert^2 d\mu} = 1 + \lambda^2 \, \frac{\langle \phi, \Delta^{-1}\phi\rangle_{\dot{H}^1(M)}}{\lVert \phi\rVert^2_{\dot{H}^1(M)}},$$
and this can be made negative for some $\phi\in \im[L]$ if and only if 
$$ \sup_{\phi\in \im[L]} \frac{\langle \phi, (-\Delta)^{-1}\phi\rangle_{\dot{H}^1(M)}}{\lVert \phi\rVert^2_{\dot{H}^1(M)}}
\ge \frac{1}{\lambda^2},$$
and the left-hand side is precisely the operator norm of $(-\Delta)^{-1}$ restricted to $\im[L]$. 
Since $(-\Delta)^{-1}$ is positive, compact, and self-adjoint on $\dot{H}^1(M)$, it is also positive, compact, and self-adjoint on the closed
subspace $\im[L]$. Hence it has a sequence of positive eigenvalues converging to zero, and the operator norm of it is the largest one. 
So the operator norm is larger than $1/\lambda^2$ if and only if there is a $c \geq 1/\lambda^2$ 
such that 
\begin{equation}\label{imageLeigenfunction}
(-\Delta)^{-1}L(f) = cL(f).
\end{equation}
\end{proof}


Solving the eigenfunction problem \eqref{imageLeigenfunction} is rather difficult since $L$ is quite far from invertible---its kernel consists of all functions constant on the level sets of $\psi$, and in particular any function $\Phi\circ \psi$ for $\Phi\colon \mathbb{R}\to\mathbb{R}$ will be in that kernel. In principle one could integrate the Green function for $\Delta^{-1}$ along the level sets to try to solve directly for $f$, or apply the operator $L^{-1}\Delta^{-1}L$ repeatedly to a Fourier basis in hopes that it converges to an eigenfunction, but computationally this becomes rather difficult. For example, inverting $L$ involves integrating around the level curves, but since there are always singular hyperbolic points where the gridlines cross, the local behavior can complicate things substantially. As such we take a more indirect approach in the next section, using higher-order Sobolev inner products.

\section{Numerics}\label{numerics}

Here we give some details on the numerics that helped provide the results of Theorems \ref{offdiagonal} and \ref{diagonal}, and explain the shape of the perturbations generating the variations that satisfy the Misiolek criterion. First we discuss the general setup for any surface $M$ possibly with boundary, then we specialize to $M=\mathbb{T}^2$. The implementation described in this section and used to generate Figures \ref{fig:perturb_wrt_p} and \ref{fig:perturb_unconstrained} is available on github: \url{https://github.com/alebrigant/conjugate-points}. 

\subsection{The minimization strategy}

Let $C_0^\infty(M)$ denote the space of $C^\infty$-functions $f:M\rightarrow\R$ that vanish on the boundary $\partial M$. We define the following operators on $C^\infty_0(M)$
\begin{equation}\label{previouslydefinedoperators} 
Lf:=\{\psi, f\}, \quad \Lambda f:=-\Delta f, \quad \Omega f:=-L(\Lambda-\lambda^2 I)Lf.
\end{equation}
The goal is to find a function $f\in C_0^\infty$ such that the corresponding Poisson bracket $\phi=\{\psi,f\}$ minimizes some normalized version of the Misio{\l}ek index \eqref{misiolekindexagain}. This index can be rewritten in terms of the previously defined operators \eqref{previouslydefinedoperators} as
\begin{align*}
MI(\phi)&=-\int_M\phi(\Delta+\lambda^2I)\phi \,d\mu=\int_M Lf(\Lambda-\lambda^2 I)Lf \,d\mu=\langle f, \Omega f\rangle_{L^2},
\end{align*}
where we have used the antisymmetry of the operator $L$ with respect to the $L^2$ inner product on $C_0^\infty(M)$. To minimize $MI(\phi)$ we need to constrain $f$; otherwise the minimum is either zero or negative infinity. The easiest way to do this is to require that some norm of $f$ be constrained to be 1, which is equivalent to choosing $f$ so that it minimizes the Rayleigh quotient
$$RQ(f)=\frac{\langle f, \Omega f\rangle_{L^2}}{\langle f, \Gamma f\rangle_{L^2}},$$
where $\Gamma$ is some positive-definite symmetric operator on $C_0^\infty(M)$. If $\Gamma$ has some relatively high Sobolev order, the minimizer we find will be smoother. By the usual calculus of variations method, minimizers of this must satisfy
$$\Omega f=c\Gamma f \quad\text{for some }c\in\R.$$
Equivalently, $f$ is an eigenfunction of the operator $\Gamma^{-1}\Omega$, and if we can find an eigenvalue $c$ which is negative, it proves the index form can be made negative, and thus that there is eventually a conjugate point along the geodesic. The idea is therefore to make $\Gamma$ strong enough as a differential operator that $\Gamma^{-1}$ more than compensates all the differential operators in $\Omega$. Choosing $\Gamma=\Lambda^p$ is equivalent to changing the inner product to the homogeneous Sobolev metric of order $p$
$$\langle f,g\rangle_{\dot H^p}=(-1)^p\int_M f\Delta^pg \,d\mu = \langle f, \Lambda^pg\rangle_{L^2}$$
so that the Rayleigh quotient becomes
\begin{equation}\label{rayleigh}
RQ(f)=\frac{\langle f, \Omega_p f\rangle_{\dot H^p}}{\langle f, f\rangle_{\dot H^p}},
\end{equation}
where 
$$\Omega_p=\Lambda^{-p}\Omega=-\Lambda^{-p}L(\Lambda-\lambda^2 I)L.$$ 
Replacing the $L^2$-inner product by the $\dot H^p$ inner product yields an operator $\Omega_p$ that is bounded for $p\geq 2$, and compact for $p\geq 3$. 

\begin{figure}
\includegraphics[width=0.24\linewidth]{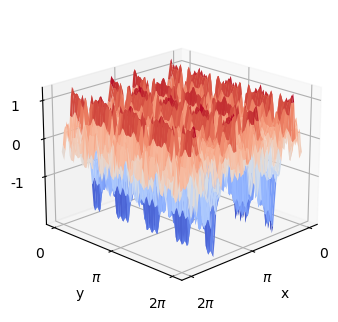}
\includegraphics[width=0.24\linewidth]{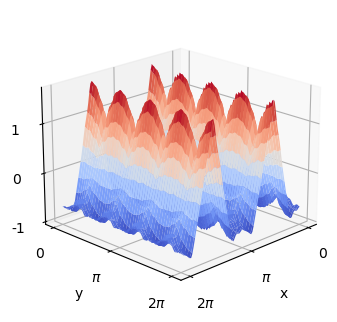}
\includegraphics[width=0.24\linewidth]{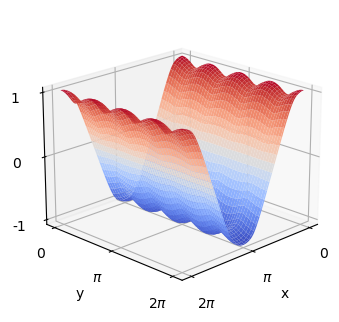}
\includegraphics[width=0.24\linewidth]{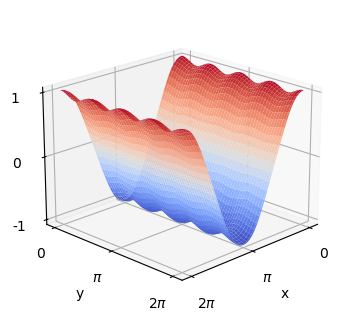}\\
\includegraphics[width=0.24\linewidth]{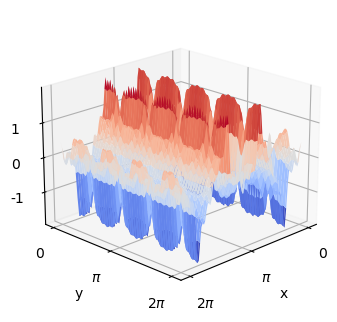}
\includegraphics[width=0.24\linewidth]{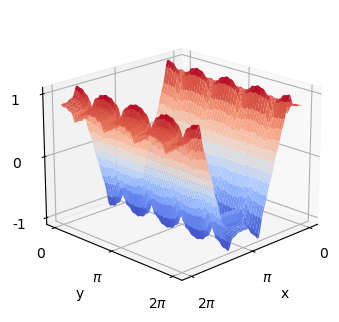}
\includegraphics[width=0.24\linewidth]{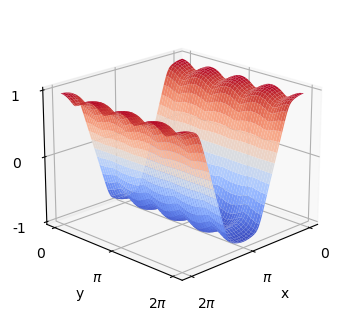}
\includegraphics[width=0.24\linewidth]{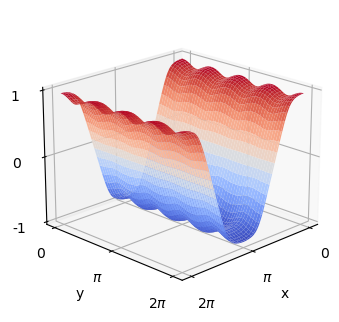}
\caption{On the top row, minimizers $f:M\rightarrow \R$ of \eqref{rayleigh} in the case $m=3$ and $n=2$ and for different values of the order $p$ of the homogeneous Sobolev norm: $0, 1, 2, 3$ from left to right. On the bottom row, minimizers in the case $m=2$ and $n=2$ for the same values of $p$.}
\label{fig:perturb_wrt_p}
\end{figure}

The problem of minimizing the Rayleigh quotient \eqref{rayleigh} can be made finite-dimensional by decomposing the unknown $f$ in the following spatial Fourier basis
\begin{equation}\label{fourier}
f(x,y)=\sum_{j=0}^N\sum_{k=-N}^N a_{jk}\cos(jx+ky).
\end{equation}
Due to the evenness of the cosine function, it is enough to take only positive $k$ indices for $j=0$ in the above double sum. Then each function $f$ can be represented by a $(N(2N+2)+1)$-size vector of real coordinates $(a_{jk})_{j,k}$, and the operator $\Omega_p$ by a square matrix of same size. The matrix representation of $\Omega_p$ is obtained by multiplying matrices that are all diagonal except for the Poisson bracket operator $L$ matrix (for which we give a formula in the Appendix, for the aid of the reader). Then the problem of minimizing \eqref{rayleigh} boils down to finding the minimal eigenvalue  of the $\Omega_p$ matrix, and the function $f$ whose Poisson bracket minimizes the Misio{\l}ek criterion is given by the corresponding eigenvector. 

\subsection{The minimizing deformations}

Figure~\ref{fig:perturb_wrt_p} shows the minimizers found in the case $m=3$ and $n=2$ (top row) and the case $m=2$ and $n=2$ (bottom row), for different values of the Sobolev order $p$. The deformations $g(t,x,y)=\sin(\frac{\pi t}{T})f(x,y)$ associated to all the spatial deformations $f$ shown in this figure induce vector fields $Y=\sgrad g$ that make the Misio{\l}ek index, and thus the index form, negative. However the ``shape'' of this minimizing deformation stabilizes only for $p\geq 2$, and for these values we see a characteristic ``V-shape'' appearing, i.e., a perturbation of $\cos(x)$, which we used to obtain formulas \eqref{foffdiag} and \eqref{fdiag}.

\begin{figure}
\includegraphics[width=0.24\linewidth]{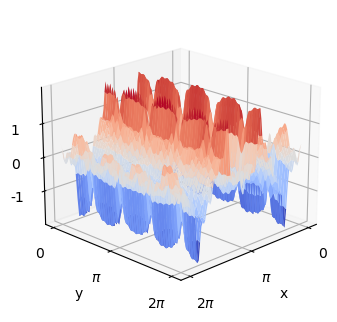}
\includegraphics[width=0.24\linewidth]{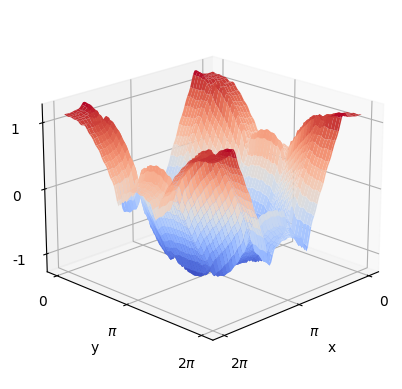}
\includegraphics[width=0.24\linewidth]{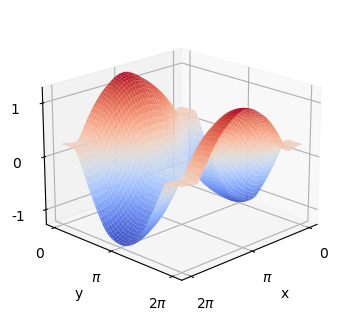}
\includegraphics[width=0.24\linewidth]{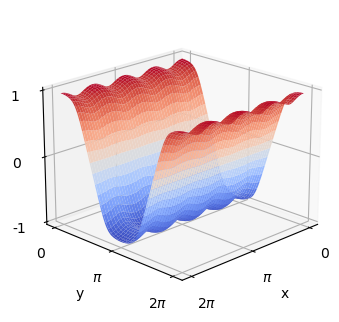}
\caption{Minimizers $f:M\rightarrow \R$ of \eqref{rayleigh} in the case $m=2$ and $n=2$ and for different values of the order $p$ of the homogeneous Sobolev norm: $0, 1, 2, 3$ from left to right.}
\label{fig:perturb_unconstrained}
\end{figure}

In the diagonal case $m=n$, due to the symmetry in the variables $x$ and $y$, another minimizer is given by a perturbation of $\cos(y)$, the symmetric perturbation obtained by exchanging $x$ and $y$. To obtain the figures of the bottom row of Figure~\ref{fig:perturb_wrt_p}, we constrained the solution to have Fourier coefficient $a_{01}=0$. Without this constraint, we find the results of Figure~\ref{fig:perturb_unconstrained}, where both minimizers appear separately as well as mixed, as in the case $p=1$. This is similar to the minimizer found by the second author in \cite{preston2023conjugate}, displayed in Figure 7.

The reason for the dominance of the term $\cos{x}$ in the functions appearing in Theorem \ref{offdiagonal}--\ref{diagonal} is that this is the lowest-eigenvalue eigenfunction of $\Lambda$. Roughly speaking, we first compute 
$$ L(\cos{x}) = n\sin{x}\cos{mx} \sin{ny},$$
then obtain
$$ (\Lambda-\lambda^2I)L(\cos{x}) = 2mn\cos{x}\sin{mx}\sin{ny}$$
plus a term with a smaller coefficient, and finally
$$ L(\Lambda-\lambda^2I)L(\cos{x}) 
= -\frac{n^2}{4} \cos{x} + S,$$
where $S$ is a linear combination of terms of the form $\cos(jx+ky)$, which are eigenfunctions of $\Lambda^{-p}$ for the eigenvalues $(j^2+k^2)^{-p}$, for integers $j$ and $k$. Therefore, after applying the inverse Laplacian $\Lambda^{-p}$ for a high power $p$, the correction terms in $S$ all become small as $p$ increases while the $\cos{x}$ term is preserved.
This explains the dominance of the V-shape: although $L(\cos{x})$ does not involve another $\cos{x}$ term, $L^2(\cos{x})$ does, and it becomes dominant after applying smoothing operators.

\section{Other open questions}\label{11sec}

We conclude with some possible other directions to pursue these techniques. An obvious one is to apply this to other steady flows on the torus that come from Laplacian eigenfunctions; the Kolmogorov flows we have discussed have rectangular cells, but one could also consider stream functions of the form
$$ \psi = \cos{(mx+ny)} + c \cos{(mx-ny)} $$
which for $c\ne 1$ have skew quadrilateral cells. One could also consider stretched tori where the eigenfunctions look like $\cos{(mx+\alpha ny)}$ for some positive parameter $\alpha$, as considered in \cite{drivas2021conjugate}.

We have seen that whenever the Kolmogorov eigenvalue is strictly larger than the minimal one $\lambda=1$, the Misio{\l}ek index can be made negative. When $\lambda=1$, the question is more delicate. There are some simple eigenfunctions of the Laplacian such as the shear flows generated by $\sin{x}$, $\cos{x}$, $\sin{y}$, and $\cos{y}$ where the fluid flows have nonpositive curvature along them, and therefore no conjugate points by Misio{\l}ek~\cite{misiolekstability}. However there are also stream functions like $\sin{x}+\sin{y}$ which are not shear flows and have nontrivial cells. The Misio{\l}ek index must clearly be nonnegative in this case, but there may still be conjugate points detectable by methods such as in \cite{preston2023conjugate}.

It is also easy to see how the same techniques could generate conjugate points on other surfaces. It would be interesting to try the same methods to find conjugate points along other eigenfunctions of the $2$-sphere: the smallest-eigenvalue function generates rigid rotations and its geodesic is known to have many conjugate points~\cite{misiolekstability}. It seems likely that every eigenfunction of the Laplacian on $S^2$ generates a geodesic with conjugate points, but no others are actually known concretely. This is connected with work of Benn~\cite{benn2021conjugate} on conjugate points along \emph{nonsteady} Rossby-Haurwitz waves on the sphere.

\section*{Appendix}

\begin{proposition}
    Let $f:M\rightarrow M$ be a function and $a_{jk}$ the coefficients of its Fourier decomposition \eqref{fourier}, for $(j,k)\in \{0\}\times\{0,\hdots,N\}\cup \{1,\hdots,N\}\times\{-N,\hdots,N\}$. Then
    \begin{align*}
    Lf=\{\psi,f\}=\frac{1}{8}\sum_{j=-\infty}^\infty\sum_{k=-\infty}^\infty\big(&(mk-nj)(A_{j-m,k-n}-A_{j+m,k+n})\\
    +&(mk+nj)(A_{j-m,k+n}-A_{j+m,k-n})\big)\cos(jx+ky)
    \end{align*}
    where
    $$A_{jk}=\begin{cases}
    a_{jk}\,\,&\text{if}\quad 0<j\leq N, \, |k|\leq N,\\
    a_{-j,-k}\,\,&\text{if}\quad -N\leq j<0, \, |k|\leq N,\\
    a_{j|k|} \,\, &\text{if}\quad j=0, \, |k|\leq N,\\
    0\,\, &\text{otherwise.}
    \end{cases}
    $$
\end{proposition}

\bibliographystyle{abbrv}
\bibliography{bibliography}

\end{document}